\documentclass[12pt]{amsart}
\usepackage{amsfonts}
\usepackage{amssymb}
\usepackage[letterpaper, left=2.5cm, right=2.5cm, top=2.5cm,
bottom=2.5cm,dvips]{geometry}
\newtheorem{theorem}{Theorem}
\theoremstyle{plain}

\newtheorem{conjecture}{Conjecture}
\newtheorem{corollary}{Corollary}

\newtheorem{lemma}{Lemma}

\numberwithin{equation}{section}

\begin{document}
	\title[Majority choosability of countable graphs]{Majority choosability of countable graphs}
	\author{Marcin Anholcer}
	\address{Faculty of Informatics and Electronic Economy, Pozna\'{n}
		University of Economics and Business,  61-875 Pozna\'{n}, Poland}
	\email{m.anholcer@ue.poznan.pl}
	\author{Bart\l omiej Bosek}
	\address{Theoretical Computer Science Department, Faculty of Mathematics and Computer Science, Jagiellonian
		University, 30-348 Krak\'{o}w, Poland}
	\email{bosek@tcs.uj.edu.pl}
	\author{Jaros\l aw Grytczuk}
	\address{Faculty of Mathematics and Information Science, Warsaw University
		of Technology, 00-662 Warsaw, Poland}
	\email{j.grytczuk@mini.pw.edu.pl}
	\thanks{Research supported by the National Science Center of Poland, grant 2015/17/B/ST1/02660.}
	
	\begin{abstract}
	In any vertex coloring of a graph some edges have differently colored ends (\emph{good} edges) and some are monochromatic (\emph{bad} edges). In a proper coloring all edges are good. In a \emph{majority coloring} it is enough that for every vertex $v$, the number of bad edges incident to $v$ does not exceed the number of good edges incident to $v$. A well known result of Lov\'{a}sz \cite{Lovasz} asserts that every finite graph has a majority $2$-coloring. A similar statement for countably infinite graphs is a challenging open problem, known as the \emph{Unfriendly Partition Conjecture}.
	
	We consider a natural list variant of majority coloring. A graph is \emph{majority $k$-choosable} if it has a majority coloring from any lists of size $k$ assigned arbitrarily to the vertices. We prove that every countable graph is majority $4$-choosable. We also consider a natural analog of majority coloring for directed graphs. We prove that every countable digraph is also majority $4$-choosable. We pose list and directed analogs of the Unfriendly Partition Conjecture, stating that every countable graph is majority $2$-choosable and every countable digraph is majority $3$-choosable.
	\end{abstract}
	
	\maketitle
\section{Introduction}
Let $G=(V,E)$ be a simple graph and let $c$ be any vertex coloring of $G$. An edge $uv\in E$ is \emph{good} in coloring $c$ if $c(u)\neq c(v)$, otherwise it is \emph{bad}. A coloring $c$ is called a \emph{majority coloring} of a graph $G$ if every vertex $v$ has at least as many good as bad incident edges.

An old theorem of Lov\'{a}sz \cite{Lovasz} asserts that every finite graph $G$ is majority $2$-colorable. The proof is very simple: just notice that any vertex $2$-coloring that minimizes the total number of bad edges in $G$ satisfies the majority condition. Indeed, if not, then there is a vertex $v$ with an excess of bad incident edges. But then switching the color of $v$ exchanges bad and good edges incident to $v$, which results in decreasing the total number of bad edges in $G$.

The majority coloring problem can be considered for infinite graphs as well. The majority condition says than that the cardinality of the set of bad edges incident to any vertex $v$ is at most the cardinality of the set of good edges incident to $v$. The question whether every countable graph has a majority $2$-coloring was posed by Cowan and Emerson \cite{Cowan Emerson} (see \cite{Aharoni}) and became known as the \emph{Unfriendly Partition Conjecture}.

\begin{conjecture}[Unfriendly Partition Conjecture \cite{Cowan Emerson}]\label{Conjecture UPC}
	Every countable graph is majority $2$-colorable.
\end{conjecture}

The conjecture can be easily verified for locally finite graphs as well as for graphs with all vertices of infinite degree. Aharoni, Milner, and Prikry \cite{Aharoni} proved it for graphs with finitely many vertices of infinite degree. Bruhn, Diestel, Georgakopoulos, and Spr\"{u}ssel \cite{Bruhn} confirmed it for rayless graphs, and Berger \cite{Berger} proved it for graphs not containing a subdivision of an infinite clique. These results are valid for arbitrary infinite graphs (not only countable). On the other hand, Milner and Shelah \cite{Shelah} proved that every infinite graph is majority $3$-colorable, and that there exist uncountable graphs demanding three colors.

We study a natural list version of majority coloring. Suppose that each vertex $v$ of a graph $G$ is given a list $L(v)$ of colors, and only colors from $L(v)$ may be used to color $v$. A graph $G$ is called \emph{majority $k$-choosable} if it has a majority coloring from arbitrary lists of size $k$. We state the following strengthening of the Unfriendly Partition Conjecture. 

\begin{conjecture}\label{Conjecture LUPC}
	Every countable graph is majority $2$-choosable.
\end{conjecture}

Notice that it is not obvious \emph{a priori} that countable graphs are majority $k$-choosable for any finite constant $k$. However, we prove in Theorem \ref{Theorem Main} that every countable graph is majority $4$-choosable.

We also consider a natural analog of this problem for directed graphs. A vertex coloring of a digraph is a \emph{majority coloring} if every vertex $v$ has at most half bad out-going edges. This idea was introduced by Kreutzer, Oum, Seymour, van der Zypen, and Wood in \cite{Kreutzer}, where it was proved that every finite digraph is majority $4$-colorable and conjectured that three colors are sufficient. We extended this result in \cite{Anholcer} by proving that every finite digraph is majority $4$-choosable. In \cite{Anholcer BG 2} we proved that every countable digraph is majority $5$-colorable, which was subsequently improved to $4$ by Bowler, Erde and Pitz \cite{Bowler EP}. This leads to the following directed analog of the Unfriendly Partition Conjecture.

\begin{conjecture}
	Every countable digraph is majority $3$-colorable.
\end{conjecture}

We prove in Theorem \ref{Theorem Main Directed} that every countable digraph is also majority $4$-choosable. This leads to the following list analog of directed version of the Unfriendly Partition Conjecture.
 
\begin{conjecture}
Every countable digraph is majority $3$-choosable.
\end{conjecture}

\section{The results}

\subsection{Undirected graphs}

It is not hard to see that any finite graph is majority $2$-choosable. Indeed, the same argument with a coloring that minimizes the total number of bad edges is sufficient. By compactness it follows that every locally finite graph is majority $2$-choosable. It is also not hard to prove that every countable graph in which all vertices have infinite degree is majority $2$-choosable. So, the main difficulty hides in the mixed case.

Let us call an infinite colored set $A$ \emph{almost monochromatic} if all elements of $A$ have the same color, except a finite number of them. We will need the following two lemmas.

\begin{lemma}\label{Lemma Infinite Sets}
	Let $V$ be a countable set, and let $N_1,N_2,\dots$ be a countable collection of infinite subsets of $V$. Suppose that each element $v\in V$ has assigned a list $L(v)$ of three colors. Then there is a choice of $2$-element sublists $L'(v)\subset L(v)$ such that for every coloring of $V$ from lists $L'(v)$ no set $N_i$ is almost monochromatic.
\end{lemma}

\begin{proof}
	Let $V=\{v_1,v_2,\dots\}$ be any numeration of the elements of the set $V$. Then each set $N_i$ is also linearly ordered by this numeration. To get a desired choice of sublists $L'(v_i)\subset L(v_i)$ we will use the "back-and-forth" argument, similar to the one asserting that the set of rationals is countable.
	
	Consider the sequence $$S=1,1,2,1,2,3,1,2,3,4,\dots=B_1B_2B_3\dots$$ consisting of finite blocks $B_j=1,2,\dots,j$ of initial positive integers. Let $s_i$ denote the $i$-th term of the sequence $S$. We will be choosing sublists $L'(v_i)$ consisting of color pairs in consecutive steps, accordingly to the sequence $S$.
	
	In the first step we enter the first set $N_1$ and choose any pair of colors from the list of the first element in $N_1$. In every subsequent step $s_i$ we enter the set $N_{s_i}$ and choose a pair of colors $P$ for the first available element so that: (1) $P$ is different than the pair $Q$ chosen the last time we have entered $N_{s_i}$, and (2) the three pairs, $P,Q$, and the pair $R$ chosen in the penultimate visit in $N_{s_i}$, have no common color. Such choice is always possible from lists of size three.
	
	Notice that after each step the number of elements in each set $N_i$ with assigned pairs of colors is finite. Also, the procedure based on the sequence $S$ guarantees that each set $N_i$ is visited infinitely many times. Hence, for each set $N_i$ the number of triples of pairs $P,Q,R$ satisfying the two properties (1) and (2) is infinite. This means that in any coloring of $V$ form lists $L'$, each set $N_i$ either contains infinitely many colors, or there exist at least two colors occurring infinitely many times in $N_i$. In both cases the set $N_i$ cannot be almost monochromatic. This proves the lemma.
	\end{proof}

Recall that a graph $G$ is \emph{$(m:k)$-choosable} if for every assignment of lists of size $m$ to the vertices of $G$ there is a choice of sublists of size $k$ such that every coloring from these sublists is proper. A similar notion can be defined analogously for majority coloring. By the above lemma we get immediately the following result.

\begin{corollary}\label{Corollary Graphs}
Every countable graph with all vertices of infinite degree is majority $(3:2)$-choosable.
\end{corollary}
\begin{proof}
	Let $G$ be a countable graph on the set $V=\{v_1,v_2,\dots\}$ with color lists $L(v_i)$, each of size three. Let $N_i=N(v_i)$ be the set of all neighbors of the vertex $v_i$ in $G$. Let $L'(v_i)\subset L(v_i)$ be any choice of $2$-element sublists satisfying the assertion of Lemma \ref{Lemma Infinite Sets}. Finally, let $c$ be any coloring from lists $L'(v_i)$. Then in each neighborhood $N_i$ there exist infinitely many vertices whose colors are different than $c(v_i)$. Thus, for every vertex $v_i$, the number of good edges incident to $v_i$ is infinite, which means that $c$ is a majority coloring.
	\end{proof}

The following lemma is crucial for our main result. It was inspired by the result of Bernardi \cite{Bernardi} generalizing the theorem of Lov\'{a}sz \cite{Lovasz}.

\begin{lemma}\label{Lemma List Bernardi}
	Let $G$ be a countable locally finite graph. Suppose that $V(G)=F\cup I$, where $I$ is an independent set in $G$. Suppose that each vertex $v\in F$ is assigned a list $L(v)$ of four colors and each color $x\in L(v)$ is assigned a real number $r_{v}(x)$. Assume that for every vertex $v\in F$ we have: 
	\begin{equation}\sum_{x\in L(v)}r_{v}(x) \geqslant 2\deg(v).
	\end{equation}
	Assume further that every vertex $u\in I$ is assigned a list $L'(u)$ of two colors. Then there is a coloring $c$ of $F$ from lists $L(v)$ such that for every coloring $c'$ of $I$ from lists $L'(u)$, the number of bad edges incident to any vertex $v\in F$ is at most $r_{v}(c(v))$.
\end{lemma}

\begin{proof}
	First we prove the assertion for finite graphs $G$. For every vertex $v\in F$ and for every color $x\in L(v)$, let $m_v(x)$ be the number of neighbors of $v$ in $I$ containing color $x$ in their lists: $$m_v(x)=\left|\{u\in I\cap N(v): x\in L'(u)\}\right|.$$ For a fixed coloring $c$ of the set $F$ from lists $L(v)$, let $B_c$ denote the total number of bad edges in the induced subgraph $G[F]$. Let $M_c=\sum_{v\in F}m_v(c(v))$.
	
	Consider now a coloring $c$ of $F$ that minimizes the following quantity:
	\begin{equation}
	B_c+M_c-\frac{1}{2}\sum_{v\in V(G)}\left(r_{v}(c(v))-\sum_{x\in L(v)\setminus \{c(v)\}} r_v(x)\right).
	\end{equation}
	We claim that this coloring $c$ satisfies the assertion of the theorem.
	
	Suppose, on the contrary, that there is a coloring $c'$ of $I$ from lists $L'(u)$ and some vertex $v\in F$ that violates the stated condition. Let $L(v)=\{a,b,c,d\}$ and suppose that $c(v)=a$. Let $n^{F}_v(x)$ denote the number of neighbors of $v$ in $F$ colored by $x$ in coloring $c$, and let $n^{I}_v(x)$ be the number of neighbors of $v$ in $I$ colored with $x$ in coloring $c'$. Let $$n_v(a)=n^F_v(a)+n^I_v(a)$$ be the total number of bad edges incident to $v$. So, by our contrary assumption this number satisfies $$n_v(a)>r_v(a).$$ Notice that $$m_v(a)\geqslant n^I_v(a),$$ hence we have also that $$n^F_v(a)+m_v(a)>r_v(a)$$ in coloring $c$. Furthermore, since lists $L'(u)$ have size two, we have $$\sum_{x\in L(v)}\left(n^F_v(x)+m_v(x)\right)\leqslant 2\deg(v).$$ By condition (2.1) it follows that there must be a color in $L(v)$, say $b$, such that $$n^F_v(b)+m_v(b)<r_v(b).$$
	
	We claim that switching the color $a$ to $b$ at vertex $v$ gives a new coloring with strictly smaller value of expression (2.2). Indeed, after switching these two colors the number
	\begin{equation}
	n^F_v(a)+m_v(a)-\frac{1}{2}(r_{v}(a)-r_v(b)-r_v(c)-r_v(d))
	\end{equation}
	becomes equal to
	\begin{equation}
	n^F_v(b)+m_v(b)-\frac{1}{2}(r_{v}(b)-r_v(a)-r_v(c)-r_v(d)),
	\end{equation}
	and this is the only change in this expression. Substituting $r_v(a)$ for $n^F_v(a)+m_v(a)$ in (2.3) and $r_v(b)$ for $n^F_v(b)+m_v(b)$ in (2.4) gives the same value equal to $$\frac{1}{2}(r_v(a)+r_v(b)+r_v(c)+r_v(d)).$$ This means that the number (2.3) was strictly bigger than the number (2.4), a contradiction with the minimality of (2.2).
	
	The assertion for infinite locally finite graphs follows by the standard compactness argument. The proof is therefore complete.
\end{proof}

We are now ready to prove our main result.

\begin{theorem}\label{Theorem Main}
Every countable graph is majority $4$-choosable.
\end{theorem}

\begin{proof}
	Let $G$ be a graph on a countable set of vertices $V$. For each vertex $v\in V$, let $N(v)$ denote the set of all neighbors of $v$ in $G$. Let $N[v]=N(v)\cup \{v\}$ be the closed neighborhood of the vertex $v$.
	
	First split the set $V$ into two parts: $F$---the set of all vertices with finite degree, and $I$---the set of all vertices with infinite degree. Further, let $A\subseteq I$ be the set consisting of all vertices $v\in I$ for which the set $N(v)\cap I$ is infinite. Let $B\subseteq I$ be the set of the remaining vertices of $I$. Notice, that for each vertex $v\in B$ the set $N(v)\cap F$ is infinite.
	
	Assume that each vertex $v\in F$ has assigned a list $L(v)$ of four colors, while each vertex $u\in I$ is assigned with a list $L(u)$ of three colors. We shall color the graph $G$ in three stages.
	
	\textbf{Stage 1.} We apply Lemma \ref{Lemma Infinite Sets} to the set $I$ and the family of infinite subsets $N[u]\cap I$, with $u\in A$. So, for each vertex $u\in I$ we choose a $2$-color sublist $L'(u)\subset L(u)$ such that no set $N[u]\cap I$ is almost monochromatic in any coloring from lists $L'$. Thus, the majority condition is already guaranteed for each vertex $u\in A$.
	
	\textbf{Stage 2.} We will color the subgraph $G[F]$ of $G$ induced by the set $F$ by using Lemma \ref{Lemma List Bernardi} with $r_v(x)=\frac{1}{2}\deg(v)$ for every vertex $v\in F$ and every color $x\in L(v)$. Hence we get a coloring $c$ from lists $L(v)$ which will satisfy the majority condition for any future coloring of the set $I$ from lists $L'$. 
	
	\textbf{Stage 3.} We color the set $I$ from lists $L'(u)$. Let $u$ be an arbitrary vertex in $I$. If $u\in A$ then we color $u$ by any color from its list since the majority condition has been already guaranteed in Stage 1. If $u\in B$ and the set $F\cap N(u)$ is not almost monochromatic, then we color $u$ by any color from its list $L'(u)$. If $u\in B$ and $F\cap N(u)$ is almost monochromatic, we choose for $u$ the opposite color from the list $L'(u)$, and the majority condition is satisfied for $u$, too.

 The proof is complete.
\end{proof}
\subsection{Directed graphs}
First notice that Lemma \ref{Lemma Infinite Sets} can be easily applied in the directed case. In particular, we get immediately the following analog of Corollary \ref{Corollary Graphs}.

\begin{corollary}\label{Corollary Directed}
Every countable digraph with all vertices of infinite out-degree is $(3:2)$-choosable.
\end{corollary}

The directed analog of Lemma \ref{Lemma List Bernardi} is slightly more problematic. Curiously, the statement is literally the same, though the proof is different. A digraph is \emph{locally finite} if all its vertices have finite out-degree.
\begin{lemma}\label{Lemma Bernardi Directed}
	Let $D$ be a countable locally finite digraph. Suppose that $V(D)=F\cup I$, where $I$ is an independent set in $D$. Suppose that each vertex $v\in F$ is assigned a list $L(v)$ of four colors and each color $x\in L(v)$ is assigned a real number $r_{v}(x)$. Assume that for every vertex $v\in F$ we have: 
\begin{equation}\sum_{x\in L(v)}r_{v}(x) \geqslant 2\deg^+(v).
\end{equation}
Assume further that every vertex $u\in I$ is assigned a list $L'(u)$ of two colors. Then there is a coloring $c$ of $F$ from lists $L(v)$ such that the for every coloring $c'$ of $I$ from lists $L'(u)$, the number of bad edges out-going from any vertex $v\in F$ is at most $r_{v}(c(v))$.
\end{lemma}
\begin{proof}
	We prove the lemma for finite digraphs (the infinite case follows by compactness). The proof is by induction on the size of the independent set $I$. If $I$ is empty, then the statement of the lemma coincides with the main result in \cite{Anholcer}.
	
	For the inductive step, let $u$ be any vertex in $I$ with the list $L'(u)=\{a,b\}$, and let $D'$ be a digraph $D$ with vertex $u$ deleted. Moreover, if $v$ is any in-coming neighbor of $u$, then we put $r'_v(a)=r_v(a)-1$ and $r'_v(b)=r_v(b)-1$. In all other cases we keep $r'_v(x)=r_v(x)$. Notice that, by condition (2.5) for the digraph $D$, we may write $$\sum_{x\in L(v)}r'_{v}(x) \geqslant \sum_{x\in L(v)}r_{v}(x)-2\geqslant2(\deg^+(v)-1)\geqslant2\deg^+_{D'}(v).$$ So, condition (2.5) is also satisfied for $D'$ and we can apply induction to get a coloring $c$ satisfying the stated condition for the digraph $D'$ with numbers $r'_v(x)$. However, the same coloring $c$ is valid for $D$ after restoring the vertex $u$, because the only colors that may be used for $u$ are the colors $a$ and $b$. This completes the proof.
	\end{proof}

Using Lemma \ref{Lemma Infinite Sets} and Lemma \ref{Lemma Bernardi Directed} one may now prove our second main result in much the same way as in the undirected case.
\begin{theorem}\label{Theorem Main Directed}
	Every countable directed graph is majority $4$-choosable.
\end{theorem}

It is enough to split the digraph into two parts accordingly to whether the out-degree of a given vertex is finite or infinite, and then color the resulting sets in three stages, like in the proof of Theorem \ref{Theorem Main}. The details are left to the reader.

\section{Final remarks}

Let us conclude the paper with two more open problem concerning the special case of acyclic countable digraphs. Let $D$ be a finite digraph not containing directed cycles. As observed in \cite{Kreutzer}, $D$ is easily majority $2$-colorable. Indeed, we may linearly order the vertices of $D$ so that all out-going neighbors of each vertex $v$ are to the left of $v$. Then we may use the greedy coloring algorithm to get a majority coloring of $D$ by choosing in each step one of the two colors that appears not more frequently than the other on the out-neighbors of the current vertex. Curiously, it is not clear what is going on for infinite acyclic digraphs.

\begin{conjecture}
	Every countable acyclic digraph is majority $2$-colorable.
\end{conjecture}

We proved in \cite{Anholcer BG 2} that the above conjecture is true with three colors. However, it does not seem to be possible to extend this result on the list variant by the methods of the present paper. On the other hand, the greedy coloring algorithm works in the list setting as well, so we know that every finite acyclic digraph is majority $2$-choosable. This prompts us to stating the following extension of the last conjecture.

\begin{conjecture}
	Every countable acyclic digraph is majority $2$-choosable.
\end{conjecture}

\end{document}